\documentclass[a4paper]{amsart}
\usepackage[pdftex]{color,graphicx}
\usepackage{amssymb}
\usepackage{amsfonts}
\usepackage{esint}
\usepackage{relsize}
\usepackage{amsmath}
\usepackage{amsrefs}
\usepackage{epsfig}
\usepackage{amsthm}
\usepackage{color}
\usepackage[]{epsfig}
\usepackage[]{pstricks}
\usepackage[utf8]{inputenc}
\usepackage{multicol}
\usepackage{accents}
\newtheorem{theorem}{Theorem}[section]

\newtheorem{lemma}{Lemma}[section]
\newtheorem{corollary}{Corollary}[section]

\theoremstyle{definition}
\newtheorem{definition}{Definition}[section]

\theoremstyle{remark}
\newtheorem{remark}{Remark}[section]
\newcommand{\R}{\mathbb{R}}
\newcommand{\h}{\mathbb{H}}
\newcommand{\s}{\mathbb{S}}

\newcommand{\C}{\mathbb{C}}

\newcommand{\ria}{\rightarrow}

\newcommand{\om}{\omega}
\newcommand{\n}{\nabla}

\newcommand{\ran}{\rangle}
\newcommand{\lan}{\langle}

\DeclareMathOperator{\hess}{Hess}

\DeclareMathOperator{\rad}{rad}

\DeclareMathOperator{\trace}{trace}

\numberwithin{equation}{section}
\title[Hopf type theorems in warped product manifolds]{Hopf type theorems for surfaces in the de Sitter-Schwarzschild and Reissner-Nordstrom manifolds}

\author{Hil\'ario Alencar \and Greg\'orio Silva Neto}

\dedicatory{Dedicated to Renato Tribuzy by the occasion of his 75th birthday}

\date{August 06, 2021}

\address{Instituto de Matemática\\
Universidade Federal de Alagoas\\
Macei\'o, AL, 57072-900, Brasil\\}
\email{hilario@mat.ufal.br}

\address{Instituto de Matemática\\
Universidade Federal de Alagoas\\ 
Macei\'o, AL, 57072-900, Brasil\\}

\email{gregorio@im.ufal.br}

\begin{document}


\footnotetext{Hil\'ario Alencar and Greg\'orio Silva Neto were partially supported by the National Council for Scientific and Technological Development - CNPq of Brazil.}

\begin{abstract}

In 1951, H. Hopf proved that the only surfaces, homeomorphic to the sphere, with constant mean curvature in the Euclidean space are the round (geometrical) spheres. These results were generalized by S. S. Chern, and then by Eschenburg and Tribuzy, for surfaces, homeomorphic to the sphere, in Riemannian manifolds with constant sectional curvature whose mean curvature function satisfies some bound on its differential. In this paper, using techniques partial differential equations in the complex plane which generalizes the notion of holomorphy, we extend these results for surfaces in a wide class of warped product manifolds, which includes, besides the classical space forms of constant sectional curvature, the de Sitter-Schwarzschild manifolds and the Reissner-Nordstrom manifolds, which are time slices of solutions of the Einstein field equations of the general relativity. 
\end{abstract}

\maketitle

\section{Introduction}
%
%
%
%
%
%

In 1951, H. Hopf, see \cite{hopf} and \cite{hopf-1000}, proved that the only surfaces with constant mean curvature in $\R^3$, homeomorphic to the sphere, are the round spheres. After 32 years, the result of Hopf was extend to three-dimensional Riemannian manifolds of constant sectional curvature in 1983 by S.-S. Chern, see \cite{chern}, proving that the only surfaces with constant mean curvature in these spaces, homeomorphic to the sphere, are the geodesic spheres. Later, in 1991, J. Eschenburg and R. Tribuzy (see Theorem 3, p. 151 of \cite{E-T}) observed that, to obtain a Hopf type result, it is not necessary the immersion to have constant mean curvature, but just that the differential of the mean curvature function satisfies some upper bound, namely


\begin{theorem}[Eschenburg-Tribuzy]\label{E-T-0}
Let $Q_c^3$ be a three-dimensional Riemannian manifold with constant sectional curvature $c\in\R.$ Let $X:\Sigma\ria Q_c^3$ be an immersed surface with mean curvature function $H.$ Assume that $\Sigma$ is homeomorphic to the sphere. If there exists a local $L^p,$ $p>2,$ function $f:\Sigma\ria\R$ such that
\begin{equation}\label{cauchy-E-T}
|dH|\leq f\sqrt{H^2-K + c},
\end{equation}
where $K$ is the Gaussian curvature of $\Sigma,$ then $X(\Sigma)$ is totally umbilical.
\end{theorem}


In this paper we generalize the Eschenburg-Tribuzy theorem for the more general class of three-dimensional Riemannian manifolds $M^3=I\times\s^2,$ where $I=(0,b)$ or $I=(0,\infty),$ with the metric 
\begin{equation}\label{warped}
\lan\cdot,\cdot\ran =dt^2+h(t)^2d\om^2,    
\end{equation}
where $h:I\rightarrow\R$ is a smooth function, called warping function, and $d\om^2$ denotes the canonical metric of the $2$-dimensional round sphere $\s^2.$ With the metric (\ref{warped}), the product $M^3=I\times\s^2$ is called a warped product manifold and generalizes the space forms with constant sectional curvature. In fact, the metrics of the space forms of constant sectional curvature $c\in\R$ can be written in polar coordinates as in \eqref{warped}, where
\[
h(t)=t \ \mbox{for} \ \R^3, \ h(t) = \dfrac{1}{\sqrt{c}}\sin(\sqrt{c}t)\ \mbox{for} \ \s^3(c),\ h(t) =\dfrac{1}{\sqrt{-c}}\sinh(\sqrt{-c}t) \ \mbox{for}\ \h^3(c).
\]
The warped product manifold $M^3$ has two different sectional curvatures which depend only on the parameter $t$, one tangent to the slices $\{t\}\times\s^2,$ denoted by $K_{\tan}(t),$ and other relative to the planes which contains the radial direction $\partial t,$ which de denote by $K_{\rad}(t).$ In terms of the warping function, we can write
\begin{equation}\label{tan-rad}
K_{\tan}(t) =\dfrac{1-h'(t)^2}{h(t)^2} \ \mbox{and} \ K_{\rad}(t)=-\frac{h''(t)}{h(t)},
\end{equation}
where $X,Y\in TM^3,\ X\perp \n t, \ Y\perp \n t.$ 

These manifolds were first introduced by Bishop and O' Neill in 1969, see \cite{B-ON}, and is having increasing importance due to its applications as model spaces in general relativity. Part of these applications comes from the metrics which are solutions of the Einstein equations, as the de Sitter-Schwarzschild metric and Reissner-Nordstrom metric, which we introduce later.

In recent years, immersions in warped product manifolds have been extensively studied, with many interesting papers in this subject, for instance see \cite{montiel-1}, \cite{Montiel}, \cite{bray-2}, \cite{ritore}, \cite{B-M}, \cite{DR}, \cite{AD-1}, \cite{AD-2}, \cite{AIR}, \cite{bessa}, \cite{BCL}, \cite{Xia-Wu-1}, \cite{Gimeno}, \cite{GIR}, \cite{sal-sal}, \cite{Xia-Wu-2}, \cite{aledo}, \cite{GL}, \cite{ASN}, \cite{GLW}, and \cite{SN}. We can also cite the book of Petersen, see \cite{petersen}, for a modern presentation of warped product manifolds and the book of Besse \cite{besse} for an introduction to general relativity and the deduction of Schwarzschild space-time from the Einstein equations.

The main result of this paper is the following generalization of Theorem \ref{E-T-0} for a class of warped product manifolds which contains the de Sitter-Schwarzschild and the Reissner-Nordstrom manifolds:

%

\begin{theorem}\label{hopf-warped}
Let $\Sigma$ be a surface, homeomorphic to the sphere, immersed in a warped product manifold $M^3=I\times\s^2$, with mean curvature function $H$. If there exists a non-negative $L^p,$ $p>2,$ function $f:\Sigma\ria\R$ such that
\begin{equation}\label{ineq-warped}
|dH + (K_{\tan}(t) - K_{\rad}(t))\nu dt|\leq f\sqrt{H^2 - K + K_{\tan}(t) - (1-\nu^2)(K_{\tan}(t)-K_{\rad}(t))},
\end{equation}
then $\Sigma$ is totally umbilical. 

Moreover, if $K_{\tan}(t)\neq K_{\rad}(t),$ except possibly for a discrete set of values $t\in I,$ and $\Sigma$ has constant mean curvature, then $\Sigma$ is a slice.
\end{theorem}
%

\begin{remark}
Actually, some additional hypothesis as (\ref{ineq-warped}) is needed in order to classify the slices as the only constant mean curvature spheres. In fact, it was observed by Brendle (see \cite{brendle-0}, Theorem 1.5, p. 250) that a result of Pacard and Xu (see \cite{pacard}, Theorem 1.1, p. 276) implies that in some warped product manifolds there are small spheres with constant mean curvature which are not umbilical.
\end{remark}

\begin{remark}
To obtain the slice in the second part of Theorem \ref{hopf-warped}, the assumption over $M^3$ that $K_{\tan}(t)\neq K_{\rad}(t),$ except possibly for a discrete set of values $t\in I,$ is necessary. In fact, if $K_{\tan}(t)=K_{\rad}(t)$ for some interval $(t_0,t_1)\subset I,$  then all the sectional curvatures of $M^3$ will depend only on $t.$ This will imply, by the classical Schur's Theorem, that $\widetilde{M}^3:=(t_0,t_1)\times\s^2$ has constant sectional curvature. In this case, there exists spheres, other than the slices, with constant mean curvature (in fact, the geodesic spheres centered in some point of $\widetilde{M}^3)$.
\end{remark}

Two of the most famous examples of warped product manifolds are the de Sitter-Schwarzschild manifolds and the Reissner-Nordstrom manifolds, which we describe below.

\begin{definition}[The de Sitter-Schwarzschild manifolds] Let $m>0,$ $c\in\R,$ and 
\[
(s_0,s_1)=\{r>0 ; 1-mr^{-1}-cr^2>0\}.
\] 
If $c\leq 0,$ then $s_1=\infty.$ If $c>0,$ assume that $cm^2<\frac{4}{27}.$ The de Sitter-Schwarzschild manifold is defined by $M^3(c)=(s_0,s_1)\times \s^2$ endowed with the metric
\[
\lan\cdot,\cdot\ran=\dfrac{1}{1-mr^{-1}-cr^2}dr^2 + r^2 d\om^2.
\]
In order to write the metric in the form (\ref{warped}), define $F:[s_0,s_1)\ria \R$ by
\[
F'(r)=\dfrac{1}{\sqrt{1-mr^{-1}-cr^2}}, \ F(s_0)=0.
\]
Taking $t=F(r),$ we can write $\lan\cdot,\cdot\ran=dt^2+h(t)^2d\om^2,$ where $h:[0,F(s_1))\ria[s_0,s_1)$ denotes the inverse function of $F.$ The function $h$ clearly satisfies
\begin{equation}\label{defi-SS}
h'(t)=\sqrt{1-mh(t)^{-1}-ch(t)^2},\ h(0)=s_0,\ \mbox{and}\ h'(0)=0.
\end{equation}
\end{definition}

For these manifolds, we have

\begin{corollary}[The de Sitter-Schwarzschild manifolds]
Let $\Sigma$ be a surface, homeomorphic to the sphere, immersed in the de Sitter-Schwarzschild manifold, with constant mean curvature. If there existes a non-negative $L^p,$ $p>2,$ function $f:\Sigma\ria\R$ such that
\[
|dt|\leq f \sqrt{H^2-K + c + \frac{m(3\nu^2-1)}{2h(t)^3}},
\]
then $\Sigma$ is a slice.

Here, $K$ is the Gaussian curvature of $\Sigma,$ $\nu=\lan \n t,N\ran$ is the angle function, and $N$ is the unitary normal vector field of $\Sigma$ in the de Sitter-Schwarzschild manifold.
\end{corollary}

\begin{definition}[The Reissner-Nordstrom manifolds]\label{ex-RN} The Reissner-Nordstrom manifold is defined by $M^3=(s_0,\infty)\times\s^2,$ with the metric
\[
\lan\cdot,\cdot\ran=\dfrac{1}{1-mr^{-1}+q^2r^{-2}}dr^2 + r^2 d\om^2,
\]
where $m>2q>0$ and $s_0=\frac{2q^2}{m-\sqrt{m^2-4q^2}}$ is the larger of the two solutions of $1-mr^{-1}+q^2r^{-2}=0.$ In order to write the metric in the form (\ref{warped}), define $F:[s_0,\infty)\ria \R$ by
\[
F'(r)=\dfrac{1}{\sqrt{1-mr^{-1}+q^2r^{-2}}}, \ F(s_0)=0.
\]
Taking $t=F(r),$ we can write $\lan\cdot,\cdot\ran=dt^2+h(t)^2d\om^2,$ where $h:[0,\infty)\ria[s_0,\infty)$ denotes the inverse function of $F.$ The function $h$ clearly satisfies
\begin{equation}\label{defi-RN}
h'(t)=\sqrt{1-mh(t)^{-1}+q^2h(t)^{-2}},\ h(0)=s_0,\ \mbox{and}\ h'(0)=0.
\end{equation}
\end{definition}

For these manifolds, we have

\begin{corollary}[The Reissner-Nordstrom manifolds]
Let $\Sigma$ be a surface, homeomorphic to the sphere, immersed in the Reissner-Nordstrom manifold, with constant mean curvature. If there existes a non-negative $L^p,$ $p>2,$ function $f:\Sigma\ria\R$ such that
\[
|dt|\leq f \sqrt{H^2-K +\frac{m(3\nu^2-1)}{2h(t)^3}+ \frac{q^2(1-2\nu^2)}{h(t)^4}},
\]
then $\Sigma$ is a slice.

Here, $K$ is the Gaussian curvature of $\Sigma,$ $\nu=\lan \n t,N\ran$ is the angle function, and $N$ is the unitary normal vector field of $\Sigma$ in the Reissner-Nordstrom manifold.
\end{corollary}

\begin{remark}
Since the warped product manifold is smooth at $t=0$ if and only if $h(0)=0, \ h'(0)=1,$ and all the even order derivatives are zero at $t=0$, i.e., $h^{(2k)}(0) = 0,\ k > 0,$ see \cite{petersen}, Proposition 1, p. 13, we can see the de Sitter-Schwarzschild manifolds and the Reissner-Nordstrom manifolds are singular at $t=0.$
\end{remark}

\section{Preliminaries}\label{pre-calc}


Let $F:\R^3\ria\R$ be a smooth and positive function. We will denote by 
\begin{equation}
M^3_F = (\R^3, \lan\cdot,\cdot\ran_F), \ \mbox{where} \ \lan\cdot,\cdot\ran_F=\dfrac{1}{F(x_1,x_2,x_3)^2}(dx_1^2+dx_2^2+dx_3^2)
\end{equation}
be the conformally flat three dimensional manifold.
Denote by $f=\log F$ and by $f_i = \dfrac{\partial f}{\partial x_i},$ $i=1,2,3.$ If $\Gamma_{ij}^k,$ $i,j,k=1,2,3,$ are the Christoffel symbols of $M^3,$ then
\begin{equation}\label{Gamma-Cris}
\begin{aligned}
\Gamma_{11}^{1}&=-f_1, \ \Gamma_{11}^{2}=f_2, \ \Gamma_{11}^{3}=f_3, \ \Gamma_{12}^{1}=\Gamma_{21}^{1}=-f_2,\ \Gamma_{12}^{2}=\Gamma_{21}^{2}=-f_1, \ \Gamma_{12}^{3}=\Gamma_{21}^{3}=0,\\
\Gamma_{13}^{1}&=\Gamma_{31}^{1}=-f_3, \ \Gamma_{13}^{2}=\Gamma_{31}^{2}=0, \ \Gamma_{13}^{3}=\Gamma_{31}^{3}=-f_1, \ \Gamma_{22}^{1}=f_1, \ \Gamma_{22}^{2}=-f_2, \ \Gamma_{22}^{3}=f_3,\\
\Gamma_{23}^{1}&=\Gamma_{32}^{1}=0, \ \Gamma_{23}^{2}=\Gamma_{32}^{2}=-f_3, \ \Gamma_{23}^{3}=\Gamma_{32}^{3}=-f_2, \ \Gamma_{33}^{1}=f_1, \ \Gamma_{33}^{2}=f_2, \ \Gamma_{33}^{3}=-f_3.
\end{aligned}
\end{equation}
Let $\{e_1,e_2,e_3\}$ be the canonical basis of $\R^3$ with the canonical metric. The canonical orthonormal frame of $M^3$ is
\[
\begin{aligned}
E_1(x_1,x_2,x_3)&=F(x_1,x_2,x_3)e_1,\\
E_2(x_1,x_2,x_3)&=F(x_1,x_2,x_3)e_2,\\ 
E_3(x_1,x_2,x_3)&=F(x_1,x_2,x_3)e_3.
\end{aligned}
\]
\begin{lemma}\label{connexion}
Let us denote by $\n$ the connection $\R^3$ with the metric $\lan\cdot,\cdot\ran_F.$ We have
\[
\begin{aligned}
\n_{E_1}E_1 &=F_2E_2 + F_3E_3, \ \n_{E_1}E_2 =-F_2E_1,\ \n_{E_1}E_3 = -F_3E_1,\\
\n_{E_2}E_1 &=-F_1E_2, \ \n_{E_2}E_2 = F_1E_1+F_3E_3, \ \n_{E_2}E_3 = -F_3E_2,\\
\n_{E_3}E_1 &=-F_1E_3, \ \n_{E_3}E_2 = - F_2E_3, \ \n_{E_3}E_3 = F_1E_1+F_2E_2,\\      
\end{aligned}
\]
where $F_i=\dfrac{\partial F}{\partial x_i},$ $i=1,2,3.$
\end{lemma}
\begin{proof}
We have
\[
\begin{aligned}
\n_{E_i}E_j&=F^2\n_{e_i}e_j + FF_ie_j\\
&=F^2(\Gamma_{ij}^{1}e_1 + \Gamma_{ij}^{2}e_2 + \Gamma_{ij}^{3}e_3) + F^2f_ie_j.\\
\end{aligned}
\]
The result then follows by replacing the values of $\Gamma_{ij}^k$ given by (\ref{Gamma-Cris}), and noticing that $f_i = F_i/F,$ $i=1,2,3.$
\end{proof}
Let $\Sigma$ be a smooth two dimensional Riemannian surface whose metric in a local coordinate system $\varphi:D\subset\R^2\to\Sigma$ is given by
\[
ds^2=E(x,y)dx^2 +2F(x,y)dxdy + G(x,y)dy^2, \ (x,y)\in D.
\]
A local coordinate system is called isothermal parameters if $E=G$ and $F=0.$ By results of Korn \cite{Korn} and Lichtenstein \cite{Lich} (for instance see the work of Chern \cite{chern-iso} for an elementary proof), if the functions $E,F,G:D\subset\R^2\to\R$ are H\"older continuous of order $0<\lambda<1,$ then every point of $D$ has a neighborhood whose local coordinates are isothermal parameters (remember that a function $f:D\subset\R^2\to\R$ is H\"older continuous of order $\lambda>0$ if $|f(x_2,y_2)-f(x_1,y_1)|\leq C r^\lambda,$ where $r=\sqrt{(y_2-y_1)^2+(x_2-x_1)^2}$). 

Identifying $\R^2$ with the complex plane $\C$ by taking $w=x+iy$ and $\bar{w}=x-iy,$ we have $dw = dx+idy$ and $d\bar{w}=dx-idy,$ which gives the rules of differentiation
\[
\frac{\partial}{\partial w} = \frac{1}{2}\left(\frac{\partial}{\partial x} - i\frac{\partial}{\partial y}\right) \ \mbox{and} \ \frac{\partial}{\partial \bar{w}} = \frac{1}{2}\left(\frac{\partial}{\partial x} + i\frac{\partial}{\partial y}\right).
\]
By using this complexification, we can write
\begin{equation}\label{met-1}
ds^2 = \lambda|dw^2 + \mu d\bar{w}|^2,
\end{equation} 
where
\[
\lambda = \frac{1}{4}(E+G+2\sqrt{EG-F^2}) \ \mbox{and} \ \mu=\frac{1}{4\lambda}(E-G +2iF). 
\]
Here, $\lambda>0$ and $|\mu|<1.$ If $(u,v)$ are isothermal coordinates for $\Sigma,$ then we can write
\[
ds^2=\alpha(u,v)(du^2+dv^2)=\alpha(z)|dz|^2,
\]
where $z=u+iv.$ 

\begin{remark}
The existence of isothermal coordinates can be also proved by applying known existence theorems for the Beltrami equation. In fact, since the change of coordinates satisfies
\begin{equation}\label{met-2}
ds^2=\alpha|dz|^2 = \alpha|z_w|^2\left|dw^2 + \frac{z_{\bar{w}}}{z_w}d\bar{w}\right|^2,
\end{equation} 
comparing \eqref{met-1} and \eqref{met-2}, there exist isothermal parameters in a neighborhood of $\Sigma$ is and only if there exists a solution of the Beltrami differential equation
\[
\frac{\partial z}{\partial\bar{w}}=\mu\frac{\partial z}{\partial w}.
\]
By using $L^p$ estimates for singular integral operators of Calder\'on and Zygmund, it can proved that the solution exists in any neighborhood where $\|\mu\|_\infty<1$ (see for instance \cite{IT}, p.20-21 and p.97).
\end{remark}


In this section, we will consider a smooth Riemannian surface $\Sigma$ and an isometric immersion $X:\Sigma\ria M_F^3.$ Taking isothermal parameters $u$ and $v$ in a neighborhood of $\Sigma$ and complexifying the parameters by taking $z=u+iv,$ we can identify this neighborhood of $\Sigma$ with a subset of $\C$ and obtain
\[
\lan X_z,X_{\bar{z}}\ran_F = \frac{\alpha(z)}{2},
\]
where $X_z = \frac{\partial X}{\partial z},$ $X_{\bar{z}} = \frac{\partial X}{\partial \bar{z}},$ and $\alpha(z)$ is the conformal factor of $\Sigma,$ i.e., $ds^2 = \alpha(z)|dz|^2$ is the metric of $\Sigma.$ In this case, the second fundamental form becomes
\[
II = P dz^2 + H \alpha|dz|^2 + \bar{P}d\bar{z}^2,
\]
where
\[
Pdz^2 = \lan \n_{X_z}X_z,N_F\ran_F
\]
is the Hopf differential of $X,$ i.e., the $(2,0)$-part of the complexified second fundamental form. For more details about the complexification and the Hopf differential, we refer to the chapter VI of the classical book of Hopf, see \cite{hopf-1000}.

\begin{remark}
Here and after, the bar over a quantity will mean the complex conjugate of the quantity.
\end{remark}

Since $X_z = \frac{1}{2}(X_u - iX_v)$ and $X_{\bar{z}}=\frac{1}{2}(X_u+iX_v),$ where $X_u = \frac{\partial X}{\partial u}$ and $X_v = \frac{\partial X}{\partial v},$ we have $X_u=X_z + X_{\bar{z}}$ and $X_v=i(X_z - X_{\bar{z}}).$ This implies
\[
X_u \times X_v = i(X_z + X_{\bar{z}})\times(X_z - X_{\bar{z}}) = 2i X_{\bar{z}}\times X_z, 
\]
where $\times$ means the usual vector product of $\R^3.$ On the other hand,
\[
\|X_u\times X_v\|_F=\sqrt{\|X_u\|_F^2\|X_v\|_F^2 - \lan X_u,X_v\ran_F^2} = 2\lan X_z,X_{\bar{z}}\ran_F = \alpha,
\]
where $\|Y\|^2_F = \lan Y,Y\ran_F, \ Y\in\R^3.$ Therefore, the unitary normal vector field of the immersion, with the canonical orientation, is given by
\begin{equation}\label{normal-00}
N_F = \dfrac{X_u\times X_v}{\|X_u\times X_v\|_F}=\frac{2i}{\alpha}X_{\bar{z}}\times X_z.
\end{equation}
We also have the following fundamental equations
\begin{equation}\label{codazzi}
\left\{
\begin{aligned}
\n_{X_z}X_z&=\frac{\alpha_z}{\alpha}X_z + PN_F\\
\n_{X_{\bar{z}}}X_z&=\frac{\alpha H}{2}N_F\\
\n_{X_{\bar{z}}}X_{\bar{z}}&= \frac{\alpha_{\bar{z}}}{\alpha}X_{\bar{z}} + \bar{P}N_F\\
\end{aligned}
\right.
\hspace{2cm}
\left\{
\begin{aligned}
\n_{X_z}N &=-HX_z - \frac{2P}{\alpha}X_{\bar{z}}\\
\n_{X_{\bar{z}}}N&= - \frac{2\bar{P}}{\alpha}X_z - HX_{\bar{z}}.\\
\end{aligned}
\right.
\end{equation}
Since 
\begin{equation}\label{hopf-diff-000}
\begin{aligned}
P&=\lan \n_{X_z}X_z,N\ran = \frac{1}{4}\lan \n_{X_u-iX_v}X_u-iX_v,N\ran\\
 &=\frac{1}{4}[\lan \n_{X_u}X_u,N\ran - \lan \n_{X_v}X_v,N\ran - i(\lan \n_{X_u}X_v,N\ran + \lan \n_{X_v}X_u,N\ran)]\\
 &=\frac{1}{4}[II(X_u,X_u) - II(X_v,X_v) - 2i II(X_u,X_v)],\\
\end{aligned}
\end{equation}
where $II$ is the second fundamental form of $\Sigma$ in $\R^3$, we have $P=0$ if and only if $II$ is umbilical. Moreover,
\[
\begin{aligned}
|P|^2&=\frac{1}{16}[(II(X_u,X_u) - II(X_v,X_v))^2 + 4II(X_u,X_v)^2]\\
     &=\frac{1}{16}[(II(X_u,X_u) + II(X_v,X_v))^2 - 4(II(X_u,X_u)II(X_v,X_v)-II(X_u,X_v)^2)]\\
     &=\frac{1}{16}[(\trace II)^2 - 4(\det II)].
\end{aligned}
\]
Since $H=\frac{1}{2}\trace II$ is the mean curvature and, by the Gauss equation $\det II = K - \overline{K}(T\Sigma),$
we have
\begin{equation}\label{hopf-norm}
|P|^2 =\frac{1}{4}(H^2-K + \overline{K}(T\Sigma)).
\end{equation}
Here $K$ is the Gaussian curvture of $\Sigma$ and $\overline{K}(T\Sigma)$ is the sectional curvature of $M_F^3$ relative to the two dimensional subspace $T\Sigma.$

In order to prove our main theorems, we shall need some computational lemmas.

\begin{lemma}\label{hopf-diff}
If $Pdz^2 = \lan \n_{X_z}X_z,N_F\ran_F$ be the Hopf differential of a conformal immersion $X:\Sigma\ria M^3_F,$ then
\[
P_{\bar{z}}=\frac{\alpha}{2}H_z + \lan\overline{R}(X_z,X_{\bar{z}})X_z,N\ran_F,
\]
where $\overline{R}$ is the curvature tensor of $M^3_F.$
\end{lemma}

\begin{proof}
We have
\[
\begin{aligned}
P_{\bar{z}}&= \frac{\partial}{\partial \bar{z}}\lan \n_{X_z}X_z,N_F\ran_F\\
& = \lan\n_{X_{\bar{z}}}\n_{X_z}X_z,N_F\ran_F + \lan\n_{X_z}X_z,\n_{X_{\bar{z}}}N_F\ran_F\\
&=\lan\bar{R}(X_z,X_{\bar{z}})X_z,N_F\ran_F + \lan\n_{X_z}\n_{X_{\bar{z}}}X_z,N_F\ran_F + \lan\n_{X_z}X_z,\n_{X_{\bar{z}}}N_F\ran_F\\
&=\lan\bar{R}(X_z,X_{\bar{z}})X_z,N_F\ran_F + \frac{\partial}{\partial \bar{z}}\left(\lan\n_{X_{\bar{z}}}X_z ,N_F\ran_F\right) - \lan\n_{X_{\bar{z}}}X_z,\n_{X_z}N\ran_F + \lan\n_{X_z}X_z,\n_{X_{\bar{z}}}N\ran_F\\
&=\lan\bar{R}(X_z,X_{\bar{z}})X_z,N_F\ran_F + \frac{\partial}{\partial z}\left(\frac{\alpha H}{2}\right)- \left\lan\frac{\alpha H}{2}N_F, - HX_z - \frac{2P}{\alpha}X_{\bar{z}}\right\ran_F\\
&\qquad + \left\lan\frac{\alpha_z}{\alpha}X_z + PN,-\frac{2\overline{P}}{\alpha}X_z - HX_{\bar{z}}\right\ran_F\\
&=\lan\bar{R}(X_z,X_{\bar{z}})X_z,N_F\ran_F + \frac{\alpha H_z}{2}.\\
\end{aligned}
\]
\end{proof}

In order to calculate the expression for $\lan\bar{R}(X_z,X_{\bar{z}})X_z,N_F\ran_F,$ we will use the following result, whose proof can be found in \cite{Eisen}, p. 98. Here we use the expression as it is written in the classical work of Kulkarni \cite{Kulk}, Proposition 2.2, p. 318. 

\begin{lemma}\label{kulk}
Let $(M,g)$ be a Riemannian manifold and $\bar{g}=e^{2\phi}g$ be a conformal metric. If $R$ and $\bar{R}$ denote the curvature tensors of $M$ with metrics $g$ and $\bar{g},$ respectively, then
\[
\begin{aligned}
\bar{R}(X,Y)Z&=R(X,Y)Z + [\hess\phi(Y,Z)-Y\phi Z\phi + \lan Y,Z\ran\|\n \phi\|^2]X\\
&-[\hess\phi(X,Z)-X\phi Z\phi + \lan X,Z\ran\|\n \phi\|^2]Y\\
&+\lan Y,Z\ran[\n_X\n\phi - (X\phi)\n\phi] - \lan X,Z\ran[\n_Y\n\phi - (Y\phi)\n\phi].
\end{aligned}
\]
Here, $\hess \phi,$ and $\n$ are, respectively, the Hessian and the connection (and the gradient) relative to the metric $g.$
\end{lemma}

\begin{lemma}\label{hess-NF}
If $X:(\Sigma,\alpha(z)|dz|^2)\ria M_F^3$ be a conformal immersion with normal vector $N_F$, then
\[
\lan\bar{R}(X_z,X_{\bar{z}})X_z,N_F\ran_F = -\frac{\alpha}{2F}\hess F (X_z,N_F),
\]
where $\hess F$ is the Euclidean hessian of $F.$
\end{lemma}

\begin{proof}
First, notice that
\begin{equation}\label{k-1}
\begin{aligned}
\lan\bar{R}(X_z,X_{\bar{z}})X_z,N_F\ran_F &= \frac{1}{8}\lan\bar{R}(X_u-iX_v,X_u+iX_v)(X_u-iX_v),N_F\ran_F \\
&=\frac{i}{4}\lan\bar{R}(X_u,X_v)X_u,N_F\ran_F + \frac{1}{4}\lan\bar{R}(X_u,X_v)X_v,N_F\ran_F.
\end{aligned}
\end{equation}
By using Lemma \ref{kulk} for $\phi=-\log F$, we have
\begin{equation}\label{k-2}
\begin{aligned}
\lan\bar{R}(X_u,X_v)X_u,N_F\ran_F&=\lan X_v,X_u\ran[\lan \n_{X_u}\n\phi,N_F\ran_F - (X_u\phi)\lan \n\phi, N_F\ran_F]\\
&\quad - \lan X_u,X_u\ran[\lan\n_{X_v}\n\phi, N_F\ran_F - (X_v\phi)\lan\n\phi,N_F\ran_F]\\
&=- \lan X_u,X_u\ran_F[\lan\n_{X_v}\n\phi, N_F\ran - (X_v\phi)\lan\n\phi,N_F\ran]\\
&=- \lan X_u,X_u\ran_F[\hess\phi(X_v,N_F) - \lan\n\phi,X_v\ran\lan\n\phi,N_F\ran]\\
&=- \alpha[\hess\phi(X_v,N_F) - \lan\n\phi,X_v\ran\lan\n\phi,N_F\ran].\\
\end{aligned}
\end{equation}
Analogously,
\begin{equation}\label{k-3}
\lan\bar{R}(X_u,X_v)X_u,N_F\ran_F= \alpha[\hess\phi(X_u,N_F) - \lan\n\phi,X_u\ran\lan\n\phi,N_F\ran].
\end{equation}
Replacing \eqref{k-2} and \eqref{k-3} in \eqref{k-1}, gives
\[
\lan\bar{R}(X_z,X_{\bar{z}})X_z,N_F\ran_F = \frac{\alpha}{2}[\hess\phi(X_z,N_F) - \lan\n\phi,X_z\ran\lan\n\phi,N_F\ran].
\]
On the other hand, since
\[
\n\phi =-\frac{\n F}{F} \ \mbox{and} \ \hess\phi (U,V)= -\frac{1}{F}\hess F(U,V) + \frac{1}{F^2}\lan\n F,U\ran\lan\n F,V\ran,
\]
we obtain
\[
\lan\bar{R}(X_z,X_{\bar{z}})X_z,N_F\ran_F=-\frac{\alpha}{2F}\hess F(X_z,N_F).
\]
\end{proof}

\begin{lemma}\label{angle-f}
If $X:(\Sigma,\alpha(z)|dz|^2)\ria M_F^3$ be a conformal immersion, then
\begin{equation}\label{angle-f-eq}
\frac{4}{\alpha (F(X))^2}\|X\|_z\|X\|_{\bar{z}} + \left(\frac{X}{\|X\|}\cdot N\right)^2=1,
\end{equation}

where $\|X\|$ denotes the Euclidean norm of $X.$
\end{lemma}
\begin{proof}
Considering the frame $\{X_z,X_{\bar{z}},N\}$ in $\R^3,$ we can write
\[
X=aX_z +bX_{\bar{z}} +cN,
\]
for smooth functions $a,b,c:\Sigma\ria\R.$ Since $X_z\cdot X_z = X_{\bar{z}}\cdot X_{\bar{z}}=0$ implies
\[
\begin{aligned}
X\cdot X_z &= a (X_z\cdot X_z) +  b (X_{\bar{z}}\cdot X_z) + c (N\cdot X_z)\\
&= b(F(X))^2\lan X_{\bar{z}},X_z\ran_F = b(F(X))^2\frac{\alpha}{2},
\end{aligned}
\]
and analogously for $X\cdot X_{\bar{z}}$ and $X\cdot N,$ we have
\[
X= \frac{2}{\alpha (F(X))^2}(X\cdot X_{\bar{z}})X_z +  \frac{2}{\alpha (F(X))^2}(X\cdot X_z)X_{\bar{z}} + (X\cdot N)N.
\]
This implies 
\[
\begin{aligned}
\|X\|^2 = X\cdot X &= \frac{8}{\alpha^2(F(X))^4}(X\cdot X_z)(X\cdot X_{\bar{z}})(X_{\bar{z}}\cdot X_z) + (X\cdot N)^2\\
&\frac{8}{\alpha^2(F(X))^4}(X\cdot X_z)(X\cdot X_{\bar{z}})\frac{\alpha (F(X))^2}{2} + (X\cdot N)^2\\
&\frac{4}{\alpha^2(F(X))^2}(X\cdot X_z)(X\cdot X_{\bar{z}}) + (X\cdot N)^2.\\
\end{aligned}
\]
By using $X\cdot X_z = \|X\|\|X\|_z$ and $X\cdot X_{\bar{z}}=\|X\|\|X\|_{\bar{z}},$ we obtain the result dividing the expression above by $\|X\|^2.$
\end{proof}

\begin{lemma}\label{appendix}
Let $X:(\Sigma,\alpha(z)|dz|^2)\ria M_F^3$ be a conformal immersion. If $\overline{K}(T\Sigma)$ denotes the sectional curvature of $M_F^3$ relative to the plane $dX(T\Sigma),$ then
\[
\overline{K}(T\Sigma) = -\|\n F\|^2 + \dfrac{4}{\alpha(z)F(X(z))}\hess F(X_z,X_{\bar{z}}),
\]
where $\|\n F\|^2=\left(\frac{\partial F}{\partial x_1}(X(z))\right)^2+\left(\frac{\partial F}{\partial x_2}(X(z))\right)^2+\left(\frac{\partial F}{\partial x_3}(X(z))\right)^2$ and $\hess F$ denotes the Euclidian hessian of $F.$
\end{lemma}
\begin{proof}

Considering the frame $\{X_z,X_{\bar{z}}\}$ in $T\Sigma,$ we have
\[
\overline{K}(T\Sigma)=\frac{\lan\overline{R}(X_z,X_{\bar{z}})X_z,X_{\bar{z}}\ran_F}{\lan X_z,X_z\ran_F\lan X_{\bar{z}},X_{\bar{z}}\ran_F - \lan X_{\bar{z}},X_z\ran_F^2}=-\frac{\lan\overline{R}(X_z,X_{\bar{z}})X_z,X_{\bar{z}}\ran_F}{\lan X_{\bar{z}},X_z\ran_F^2},
\]
i.e.,
\begin{equation}\label{app-2}
\overline{K}(T\Sigma) = -\frac{4}{\alpha(z)^2}\lan\overline{R}(X_z,X_{\bar{z}})X_z,X_{\bar{z}}\ran_F.
\end{equation}
Since
\begin{equation}\label{app-3}
\begin{aligned}
\lan\overline{R}(X_z,X_{\bar{z}})X_z,X_{\bar{z}}\ran_F&=\frac{1}{16}\lan\overline{R}(X_u-iX_v,X_u+iX_v)(X_u-iX_v),X_u+iX_v\ran_F\\
&=-\frac{1}{4}\lan\overline{R}(X_u,X_v)X_u,X_v\ran_F,
\end{aligned}
\end{equation}
by using Lemma \ref{kulk} for $\phi=-\log F$, we obtain
\[
\begin{aligned}
\lan\overline{R}(X_u,X_v)X_u,X_v\ran_F&=-[\hess\phi (X_u,X_u) - \lan X_u,\n \phi\ran^2 + \lan X_u,X_u\ran|\n \phi|^2]\lan X_v,X_v\ran_F\\
&\quad - \lan X_u,X_u\ran[\lan\n_{X_v}\n\phi,X_v\ran_F - \lan X_u,\n\phi\ran\lan\n\phi,X_v\ran_F]\\
&=-\alpha[\hess\phi (X_u,X_u) - \lan X_u,\n \phi\ran^2 + \lan X_u,X_u\ran|\n \phi|^2]\\
&\quad -\alpha[\hess\phi (X_v,X_v) - \lan X_v,\n \phi\ran^2].
\end{aligned}
\]
Since
\[
\n\phi = \frac{1}{F}\n F,\ \hess \phi (U,V)=-\frac{1}{F}\hess F(U,V) + \frac{1}{F^2}\lan\n F,U\ran\lan\n F,V\ran
\]
and $\lan X_u,X_u\ran = F^2\lan X_u,X_u\ran_F=F^2\alpha,$ we have
\[
\lan\overline{R}(X_u,X_v)X_u,X_v\ran_F=-\alpha\left[-\frac{1}{F}\left(\hess F(X_u,X_u)+\hess F(X_v,X_v))\right) + \alpha\|\n F\|^2\right].
\]
On the other hand,
\[
\hess F(X_z,X_{\bar{z}}) = \frac{1}{4}\hess F(X_u-iX_v,X_u+iX_v) = \frac{1}{4}[\hess F(X_u,X_u) + \hess F(X_v,X_v)],
\]
which gives
\begin{equation}\label{app-4}
\lan\overline{R}(X_u,X_v)X_u,X_v\ran_F = - \alpha^2\|\n F\|^2 + \frac{4\alpha}{F}\hess F(X_z,X_{\bar{z}}).
\end{equation}
The result then comes by replacing \eqref{app-4} in \eqref{app-3} and then in \eqref{app-2}.
\end{proof}


%

\section{Proof of the main result}


Warped product manifolds can be seen as conformally flat Riemannian manifolds with radial weight, as follows. By taking the spherical coordinates in $\R^3,$ we obtain
\[
dx_1^2+dx_2^2+dx_3^2 = dr^2 + r^2d\om^2,
\]
where $r=\sqrt{x_1^2+x_2^2+x_3^2}$ and $d\om^2$ is the canonical metric of the round sphere $\s^2.$ Let 
\[
A(r_0,r_1)=\{(x_1,x_2,x_3)\in\R^3;r_0^2\leq x_1^2+x_2^2+x_3^2\leq r_1^2\}.
\] 
If $F:A(r_0,r_1)\subset\R^3\ria\R$ is a radial function, then there exists a positive real function $u:(r_0,r_1)\subset\R\ria\R$ such that $F(x_1,x_2,x_3)=u(r).$ In this case, we have
\[
\lan\cdot,\cdot\ran_F = \dfrac{1}{u(r)^2}(dr^2+r^2d\om^2).
\]
Define $G:(r_0,r_1)\ria \R$ by $G'(r)=1/u(r).$ Since $G'(r)>0,$ we have that the function $G$ is invertible. Let $G^{-1}:I\subset \R\ria (r_0,r_1)$ be the inverse function of $G,$ where $I=G((r_0,r_1)),$ and denote by $t=G(r).$ Defining
\[
h(t)=\frac{G^{-1}(t)}{u(G^{-1}(t))},
\]
we have
\begin{equation}\label{change}
\dfrac{dr}{u(r)}=dt \ \mbox{and} \ \dfrac{r}{u(r)}=h(t).
\end{equation}
The metric $\lan\cdot,\cdot\ran_F$ thus becomes the warped metric
\begin{equation}\label{warped-11}
\lan\cdot,\cdot\ran = dt^2 + h(t)^2d\om^2,
\end{equation}
where $h:I\subset\R\ria\R$ is a smooth function, called warping function, and $M^3_F$ can be seen as the product $M^3=I\times\s^2$ with the metric (\ref{warped-11}), where $I\subset\R$ is an interval.

To prove our main theorem, we will need the following Lemma, which is an adaptation of a result due to Eschenburg and Tribuzy, see \cite{E-T} (see also the main Lemma of \cite{A-dC-T}):

\begin{lemma}[Eschenburg-Tribuzy, \cite{E-T}]\label{E-T}
Let $Q:U\subset\mathbb{C}\to\mathbb{C}$ be a complex function defined in an open set $U$ of the complex plane. Assume that
\[
|Q_{\bar{z}}| \leq f(z)|Q(z)|
\]
where $f\in L^p,$ $p>2,$ is a continuous, non-negative real function. Assume further that $z=z_0\in U$ is a zero of $Q$. Then either $Q\equiv 0$ in a neighborhood $V\subset U$ of $z_0$ or
\[
Q(z)=(z-z_0)^k Q_k(z), \ z\in V, \ k\geq 1,
\]
where $Q_k(z)$ is a continuous function with $Q_k(z_0)\neq0$.
\end{lemma}
This lemma has the following consequence for surfaces homeomorphic to the sphere. The argument is contained in the proof of the main theorem of \cite{A-dC-T}, and we include a proof here for the sake of completeness.
\begin{lemma}\label{A-dC-T}
Let $\Sigma$ be a Riemann surface homeomorphic to the sphere. Let $Qdz^2$ denote a complex quadratic differential on $\Sigma.$ Assume that
\begin{equation}\label{cauchy}
|Q_{\bar{z}}|\leq f_0|Q|,
\end{equation}
where $f_0:\Sigma\ria\R$ is a non-negative and $L^p$ function, $p>2,$ and $z$ is a local conformal parameter. Then $Q\equiv0$ in $\Sigma.$ 
\end{lemma}
\begin{proof}
Let $U\subset\Sigma$ be an open set covered by isothermal coordinates. Assume that the set of zeros of $Q$ in $U$ is not empty and let $z_0\in U$ be a zero of $Q$. By the Lemma \ref{E-T}, either $Q$ is identically zero in a neighborhood $V$ of $z_0$ or this zero is isolated and the index of a direction field determined by $Im[Qdz^2] = 0$ is $-k/2$ (hence negative). If, for some coordinate neighborhood V of zero, $Q\equiv0$, this will be so for the whole $\Sigma$, otherwise, the zeroes on the boundary of $V$ will contradict Lemma \ref{E-T}. So if $Q$ is not identically zero, all zeroes are isolated and have negative indices. Since $\Sigma$ has genus zero, the sum of the indices of the singularities of any field of directions is 2 (hence positive) by the Poincar\'e Index Theorem. This contradiction shows that $Q$ is identically zero. Notice also that $Q$ must have a zero by the Poincar\'e Index theorem, since the sum of the index is $2$ (hence nonzero).
\end{proof}

\begin{proof}[Proof of Theorem \ref{hopf-warped}]
Differentiating the second equation of (\ref{change}) relative to $r,$ and by using the first one, we have
\[
r=u(r)h(t) \Rightarrow 1= u'(r)h(t) + u(r)h'(t)\frac{dt}{dr}\Rightarrow 1 = u'(r)h(t) + h'(t)
\]
which implies
\begin{equation}\label{warped-1}
u'(r)=\frac{1-h'(t)}{h(t)}.
\end{equation}
Differentiating (\ref{warped-1}) relative to $r$ and by using the first equation of (\ref{change}), we obtain
\[
u''(r) = \frac{d}{dt}\left(\frac{1-h'(t)}{h(t)}\right)\dfrac{dt}{dr} = \left(\dfrac{-h''(t)h(t) - (1-h'(t))h'(t)}{h(t)^2}\right)\frac{1}{u(r)},
\]
i.e.,
\begin{equation}
u''(r)u(r)= - \dfrac{h''(t)}{h(t)} - \dfrac{(1-h'(t))h'(t)}{h(t)^2}.
\end{equation}
On the other hand, for $F(x_1,x_2,x_3)=u(r),$ where $r=\sqrt{x_1^2+x_2^2+x_3^2},$ we have
\[
\frac{\partial F}{\partial x_i} = x_i\frac{u'(r)}{r}  \Rightarrow \frac{\partial^2 F}{\partial x_i \partial x_j} = \dfrac{u'(r)}{r}\delta_{ij} + \dfrac{x_ix_j}{r^2}\left(u''(r) - \frac{u'(r)}{r}\right).
\]
This implies, for $u,v\in\R^3,$
\[
\hess F(u,v) = \dfrac{u'(r)}{r}(u\cdot v) + \dfrac{1}{r^2}\left(u''(r) - \frac{u'(r)}{r}\right)(X\cdot u)(X\cdot v).
\]
We observe that
\begin{equation}\label{r-t}
\frac{\partial}{\partial t} = \frac{\partial}{\partial r}\frac{dr}{dt} = u(r)\frac{\partial}{\partial r}
\end{equation}
and $N_F = u(r)N$ gives
\[
\begin{aligned}
\nu=\left\lan \frac{\partial}{\partial t},N_F \right\ran &= \left\lan u(r)\frac{\partial}{\partial r},u(r)N \right\ran_F = u(r)^2\left[\frac{1}{u(r)^2}\left(\frac{\partial}{\partial r}\cdot N\right)\right]\\
&=\left(\frac{X}{\|X\|}\cdot N\right),
\end{aligned}
\]
provided $\partial/\partial r = X/\|X\|.$ Since 
\[
X\cdot X_z = \|X\|(\|X\|)_z = rr_z,\ X\cdot X_{\bar{z}} = \|X\|(\|X\|)_{\bar{z}} = rr_{\bar{z}},
\]
we can rewrite the equation \eqref{angle-f-eq} of Lemma \ref{angle-f} as
\begin{equation}\label{nu-r}
\dfrac{4}{\alpha u(r)^2}r_zr_{\bar{z}}+\nu^2=1.
\end{equation}
This implies, 
\[
\begin{aligned}
\frac{4}{\alpha F}\hess F(X_z,X_{\bar{z}}) &=\frac{4}{\alpha u(r)}\frac{u'(r)}{r}(X_z\cdot X_{\bar{z}}) + \frac{4}{\alpha u(r)r^2}\left(u''(r) - \frac{u'(r)}{r}\right)(X\cdot X_z)(X\cdot X_{\bar{z}})\\
&=\frac{4u'(r)u(r)}{r}\frac{\lan X_z,X_{\bar{z}}\ran_F}{\alpha} + \frac{4}{\alpha u(r)r^2}\left(u''(r) - \frac{u'(r)}{r}\right)(rr_z)(rr_{\bar{z}})\\
&= \frac{2u'(r)u(r)}{r} + \left(u''(r)u(r) - \frac{u(r)u'(r)}{r}\right)\frac{4}{\alpha u(r)^2}r_zr_{\bar{z}}\\
&=\frac{2u'(r)u(r)}{r} + \left(u''(r)u(r) - \frac{u(r)u'(r)}{r}\right)(1-\nu^2)\\
&=\frac{2(1-h'(t))}{h(t)^2} - \left(\frac{h''(t)}{h(t)} + \frac{(1-h'(t))h'(t)}{h(t)^2} + \frac{1-h'(t)}{h(t)^2}\right)(1-\nu^2)\\
&=\frac{2(1-h'(t))}{h(t)^2} - \left(\frac{h''(t)}{h(t)} + \frac{1-h'(t)^2}{h(t)^2}\right)(1-\nu^2).
\end{aligned}
\]
Since, by using (\ref{warped-1}),
\[
\|\n F\|^2 = u'(r)^2 = \frac{(1-h'(t))^2}{h(t)^2},
\]
we have
\[
\begin{aligned}
-\|\n F\|^2 + \frac{4}{\alpha F}\hess F(X_z,X_{\bar{z}}) &= -\dfrac{(1-h'(t))^2}{h(t)^2} + \frac{2(1-h'(t))}{h(t)^2} - \left(\frac{h''(t)}{h(t)} + \frac{1-h'(t)^2}{h(t)^2}\right)(1-\nu^2)\\
&=\dfrac{1-h'(t)^2}{h(t)^2}- \left(\frac{h''(t)}{h(t)} + \frac{1-h'(t)^2}{h(t)^2}\right)(1-\nu^2)\\
&=K_{\tan}(t) - (K_{\tan}(t) - K_{\rad}(t))(1-\nu^2).
\end{aligned}
\]
Thus, by using Lemma \ref{appendix}, p. \pageref{appendix}, and the the equation \eqref{hopf-norm}, p.\pageref{hopf-norm}, we have
\begin{equation}\label{P-000}
\begin{aligned}
|P| &= \frac{1}{2}\sqrt{H^2 - K + \overline{K}(T\Sigma)}\\
    &= \frac{1}{2}\sqrt{H^2-K  -\|\n F\|^2 + \frac{4}{\alpha F}\hess F(X_z,X_{\bar{z}})}\\
    &= \frac{1}{2}\sqrt{H^2-K + K_{\tan}(t) - (K_{\tan}(t) - K_{\rad}(t))(1-\nu^2)}.\\
\end{aligned}
\end{equation}
On the other hand, since 
\[
\dfrac{\partial r}{\partial z} = \dfrac{dr}{dt}\dfrac{\partial t}{\partial z}=u(r)\dfrac{\partial t}{\partial z},
\]
we have
\begin{equation}\label{Pz-000}
\begin{aligned}
\hess F(X_z,N)&= \dfrac{1}{r^2}\left(u''(r) - \dfrac{u'(r)}{r}\right)(X\cdot X_z)(X\cdot N)\\
&=\left(u''(r)u(r) - \frac{u'(r)u(r)}{r}\right)\dfrac{\nu}{u(r)}r_z\\
&=-\left(\frac{h''(t)}{h(t)}+\frac{1-h'(t)^2}{h(t)^2}\right)\dfrac{\nu}{u(r)}r_z\\
&=-(K_{\tan}(t) - K_{\rad}(t))\dfrac{\nu}{u(r)}r_z\\
&=-(K_{\tan}(t) - K_{\rad}(t))\nu \dfrac{\partial t}{\partial z}.
\end{aligned}
\end{equation}
Replacing \eqref{Pz-000} in \eqref{hess-NF} and then in \eqref{hopf-diff}, we obtain
\[
P_{\bar{z}}=\frac{\alpha}{2}H_z -\frac{\alpha}{2}\hess F(X_z,N) = \frac{\alpha}{2}[H_z+(K_{\tan}(t)-K_{\rad}(t))\nu t_z]. 
\]
Since 
\[
\begin{aligned}
|P_{\bar{z}}|&=\frac{\alpha}{2}|H_z + (K_{\tan}(t) - K_{\rad}(t))\nu t_z|\\
& = \frac{\alpha}{2}|dH(X_z) + (K_{\tan}(t) - K_{\rad}(t))\nu dt(X_z)|\\
&\leq \frac{\alpha}{2}|dH + (K_{\tan}(t) - K_{\rad}(t))\nu dt|\|X_z\|\\
&= \left(\frac{\alpha}{2}\right)^{3/2}|dH + (K_{\tan}(t) - K_{\rad}(t))\nu dt|,
\end{aligned}
\]
the hypothesis \eqref{ineq-warped} and \eqref{P-000} imply
\[
|P_{\bar{z}}|\leq  2\left(\frac{\alpha}{2}\right)^{3/2}f|P|.
\]
This implies, by using Lemma \ref{A-dC-T}, p. \pageref{A-dC-T}, for $f_0=2\left(\frac{\alpha}{2}\right)^{3/2}\!\!f$, that $P=0$ everywhere in $\Sigma.$ Therefore, by \eqref{hopf-diff-000}, $\Sigma$ is umbilic.

If $H$ is constant, then $|P|\equiv 0$ gives
\[
|K_{\tan}(t) - K_{\rad}(t)||\nu||dt|\equiv 0.
\]
This and the hypothesis that $K_{\tan}(t)\neq K_{\rad}(t),$ except possibly by a discret set of values $t\in I,$ gives $|K_{\tan}(t) - K_{\rad}(t)|=0$ only for discrete set, which implies, by continuity, that 
\[
|\nu||dt|\equiv 0.
\]
Let $D_1=\{z\in\Sigma; \nu=0\}$ and $D_2 = \{z\in\Sigma; dt=0\}.$ We have $D_1\cup D_2=\Sigma.$ Observe that, by using that
\[
r_z = \frac{\partial r}{\partial z} = \frac{dr}{dt}\frac{\partial t}{\partial z} = u(r)t_z
\]
and analogously for $r_{\bar{z}},$ in (\ref{nu-r}), we have
\[
\frac{4}{\alpha}t_zt_{\bar{z}} + \nu^2=1.
\]
This gives that $\nu^2=1$ in $D_2.$ Thus, by continuity of $\nu,$ we have $D_1\cap D_2=\emptyset,$ and $D_1=\emptyset$ or $D_1=\Sigma.$ But, since $\Sigma$ is compact, there exist $t_0,t_1\in I$ such that $\Sigma\subset[t_0,t_1]\times \s^2.$ Taking the least value of $t_1$ with this property, we have that $\Sigma$ is tangent to the slice $\{t_1\}\times\s^2$ and, at this point of tangency, we have $\nu^2=1,$ i.e., $D_2\neq 0.$ Thus $D_1=\emptyset$ and $D_2=\Sigma,$ which implies $dt=0$ everywhere and this gives that $t$ is constant, i.e., $X(\Sigma)$ is a slice.


\end{proof}



\begin{bibdiv}
\begin{biblist}

\bib{aledo}{article}{
   author={Aledo, Juan A.},
   author={Rubio, Rafael M.},
   title={Stable minimal surfaces in Riemannian warped products},
   journal={J. Geom. Anal.},
   volume={27},
   date={2017},
   number={1},
   pages={65--78},
   issn={1050-6926},
   review={\MR{3606544}},
   doi={10.1007/s12220-015-9673-8},
}

\bib{ASN}{article}{
   author={Alencar, Hil\'{a}rio},
   author={Silva Neto, Greg\'{o}rio},
   title={Isoperimetric inequalities and monotonicity formulas for
   submanifolds in warped products manifolds},
   journal={Rev. Mat. Iberoam.},
   volume={34},
   date={2018},
   number={4},
   pages={1821--1852},
   issn={0213-2230},
   review={\MR{3896251}},
   doi={10.4171/rmi/1045},
}
		
\bib{AD-1}{article}{
   author={Al\'{i}as, Luis J.},
   author={Dajczer, Marcos},
   title={Uniqueness of constant mean curvature surfaces properly immersed
   in a slab},
   journal={Comment. Math. Helv.},
   volume={81},
   date={2006},
   number={3},
   pages={653--663},
   issn={0010-2571},
   review={\MR{2250858}},
   doi={10.4171/CMH/68},
}

\bib{AD-2}{article}{
   author={Al\'{i}as, Luis J.},
   author={Dajczer, Marcos},
   title={Constant mean curvature hypersurfaces in warped product spaces},
   journal={Proc. Edinb. Math. Soc. (2)},
   volume={50},
   date={2007},
   number={3},
   pages={511--526},
   issn={0013-0915},
   review={\MR{2360513}},
   doi={10.1017/S0013091505001069},
}

\bib{AIR}{article}{
   author={Al\'{i}as, Luis J.},
   author={Impera, Debora},
   author={Rigoli, Marco},
   title={Hypersurfaces of constant higher order mean curvature in warped
   products},
   journal={Trans. Amer. Math. Soc.},
   volume={365},
   date={2013},
   number={2},
   pages={591--621},
   issn={0002-9947},
   review={\MR{2995367}},
   doi={10.1090/S0002-9947-2012-05774-6},
}

\bib{A-dC-T}{article}{
   author={Alencar, Hilario},
   author={do Carmo, Manfredo},
   author={Tribuzy, Renato},
   title={A theorem of Hopf and the Cauchy-Riemann inequality},
   journal={Comm. Anal. Geom.},
   volume={15},
   date={2007},
   number={2},
   pages={283--298},
   issn={1019-8385},
   review={\MR{2344324}},
}




\bib{bessa}{article}{
   author={Bessa, G. P.},
   author={Garc\'{i}a-Mart\'{i}nez, S. C.},
   author={Mari, L.},
   author={Ramirez-Ospina, H. F.},
   title={Eigenvalue estimates for submanifolds of warped product spaces},
   journal={Math. Proc. Cambridge Philos. Soc.},
   volume={156},
   date={2014},
   number={1},
   pages={25--42},
   issn={0305-0041},
   review={\MR{3144209}},
   doi={10.1017/S0305004113000443},
}

\bib{besse}{book}{
   author={Besse, Arthur L.},
   title={Einstein manifolds},
   series={Ergebnisse der Mathematik und ihrer Grenzgebiete (3) [Results in
   Mathematics and Related Areas (3)]},
   volume={10},
   publisher={Springer-Verlag, Berlin},
   date={1987},
   pages={xii+510},
   isbn={3-540-15279-2},
   review={\MR{867684}},
   doi={10.1007/978-3-540-74311-8},
}

\bib{BCL}{article}{
   author={Bezerra, K. S.},
   author={Caminha, A.},
   author={Lima, B. P.},
   title={On the stability of minimal cones in warped products},
   journal={Bull. Braz. Math. Soc. (N.S.)},
   volume={45},
   date={2014},
   number={3},
   pages={485--503},
   issn={1678-7544},
   review={\MR{3264802}},
   doi={10.1007/s00574-014-0059-5},
}

\bib{B-M}{article}{
   author={Bray, Hubert},
   author={Morgan, Frank},
   title={An isoperimetric comparison theorem for Schwarzschild space and
   other manifolds},
   journal={Proc. Amer. Math. Soc.},
   volume={130},
   date={2002},
   number={5},
   pages={1467--1472},
   issn={0002-9939},
   review={\MR{1879971}},
   doi={10.1090/S0002-9939-01-06186-X},
}

%
\bib{bray-2}{article}{
   author={Bray, Hubert L.},
   title={Proof of the Riemannian Penrose inequality using the positive mass theorem},
   journal={J. Differential Geom.},
   volume={59},
   date={2001},
   number={2},
   pages={177--267},
   issn={0022-040X},
   review={\MR{1908823}},
}

\bib{B-ON}{article}{
   author={Bishop, R. L.},
   author={O'Neill, B.},
   title={Manifolds of negative curvature},
   journal={Trans. Amer. Math. Soc.},
   volume={145},
   date={1969},
   pages={1--49},
   issn={0002-9947},
   review={\MR{0251664}},
   doi={10.2307/1995057},
}

\bib{brendle-0}{article}{
   author={Brendle, Simon},
   title={Constant mean curvature surfaces in warped product manifolds},
   journal={Publ. Math. Inst. Hautes \'{E}tudes Sci.},
   volume={117},
   date={2013},
   pages={247--269},
   issn={0073-8301},
   review={\MR{3090261}},
   doi={10.1007/s10240-012-0047-5},
}






\bib{chern-iso}{article}{
   author={Chern, Shiing-shen},
   title={An elementary proof of the existence of isothermal parameters on a
   surface},
   journal={Proc. Amer. Math. Soc.},
   volume={6},
   date={1955},
   pages={771--782},
   issn={0002-9939},
   review={\MR{74856}},
   doi={10.2307/2032933},
}

\bib{chern}{article}{
   author={Chern, Shiing Shen},
   title={On surfaces of constant mean curvature in a three-dimensional
   space of constant curvature},
   conference={
      title={Geometric dynamics},
      address={Rio de Janeiro},
      date={1981},
   },
   book={
      series={Lecture Notes in Math.},
      volume={1007},
      publisher={Springer, Berlin},
   },
   date={1983},
   pages={104--108},
   review={\MR{730266}},
   doi={10.1007/BFb0061413},
}



\bib{DR}{article}{
   author={Dajczer, Marcos},
   author={Ripoll, Jaime},
   title={An extension of a theorem of Serrin to graphs in warped products},
   journal={J. Geom. Anal.},
   volume={15},
   date={2005},
   number={2},
   pages={193--205},
   issn={1050-6926},
   review={\MR{2152479}},
   doi={10.1007/BF02922192},
}




\bib{Eisen}{book}{
   author={Eisenhart, Luther Pfahler},
   title={Riemannian Geometry},
   note={2d printing},
   publisher={Princeton University Press, Princeton, N. J.},
   date={1949},
   pages={vii+306},
   review={\MR{0035081}},
}

\bib{E-T}{article}{
   author={Eschenburg, J.-H.},
   author={Tribuzy, R.},
   title={Conformal mappings of surfaces and Cauchy-Riemann inequalities},
   conference={
      title={Differential geometry},
   },
   book={
      series={Pitman Monogr. Surveys Pure Appl. Math.},
      volume={52},
      publisher={Longman Sci. Tech., Harlow},
   },
   date={1991},
   pages={149--170},
   review={\MR{1173039}},
}
\bib{GIR}{article}{
   author={Garc\'{i}a-Mart\'{i}nez, Sandra C.},
   author={Impera, Debora},
   author={Rigoli, Marco},
   title={A sharp height estimate for compact hypersurfaces with constant
   $k$-mean curvature in warped product spaces},
   journal={Proc. Edinb. Math. Soc. (2)},
   volume={58},
   date={2015},
   number={2},
   pages={403--419},
   issn={0013-0915},
   review={\MR{3341446}},
   doi={10.1017/S0013091514000157},
}

\bib{Gimeno}{article}{
   author={Gimeno, Vicent},
   title={Isoperimetric inequalities for submanifolds. Jellett-Minkowski's
   formula revisited},
   journal={Proc. Lond. Math. Soc. (3)},
   volume={110},
   date={2015},
   number={3},
   pages={593--614},
   issn={0024-6115},
   review={\MR{3342099}},
   doi={10.1112/plms/pdu053},
}
\bib{GLW}{article}{
   author={Guan, Pengfei},
   author={Li, Junfang},
   author={Wang, Mu-Tao},
   title={A volume preserving flow and the isoperimetric problem in warped
   product spaces},
   journal={Trans. Amer. Math. Soc.},
   volume={372},
   date={2019},
   number={4},
   pages={2777--2798},
   issn={0002-9947},
   review={\MR{3988593}},
   doi={10.1090/tran/7661},
}

\bib{GL}{article}{
   author={Guan, Pengfei},
   author={Lu, Siyuan},
   title={Curvature estimates for immersed hypersurfaces in Riemannian
   manifolds},
   journal={Invent. Math.},
   volume={208},
   date={2017},
   number={1},
   pages={191--215},
   issn={0020-9910},
   review={\MR{3621834}},
   doi={10.1007/s00222-016-0688-y},
}

\bib{hopf}{article}{
   author={Hopf, Heinz},
   title={\"{U}ber Fl\"{a}chen mit einer Relation zwischen den Hauptkr\"{u}mmungen},
   language={German},
   journal={Math. Nachr.},
   volume={4},
   date={1951},
   pages={232--249},
   issn={0025-584X},
   review={\MR{0040042}},
   doi={10.1002/mana.3210040122},
}
\bib{hopf-1000}{book}{
   author={Hopf, Heinz},
   title={Differential geometry in the large},
   series={Lecture Notes in Mathematics},
   volume={1000},
   note={Notes taken by Peter Lax and John Gray;
   With a preface by S. S. Chern},
   publisher={Springer-Verlag, Berlin},
   date={1983},
   pages={vii+184},
   isbn={3-540-12004-1},
   review={\MR{707850}},
   doi={10.1007/978-3-662-21563-0},
}
\bib{Korn}{article}{
   author={Korn, Arthur},
   title={Zwei Anwendungen der Methode der sukzessiven Ann\"aherungen},
   journal={Schwarz-Festschr.},
   date={1914},
   language={German}
   pages={215--229},
   review={ Zbl 45.0568.01},
}

\bib{Kulk}{article}{
   author={Kulkarni, Ravindra Shripad},
   title={Curvature and metric},
   journal={Ann. of Math. (2)},
   volume={91},
   date={1970},
   pages={311--331},
   issn={0003-486X},
   review={\MR{257932}},
   doi={10.2307/1970580},
}

\bib{Lich}{article}{
   author={Lichtenstein, L.},
   title={Zur Theorie der konformen Abbildung nichtanalytischer, singularit\"atenfreier Fl\"achenst\"ucke auf ebene Gebiete},
   journal={Krak. Anz.},
   language={German}
   date={1916},
   pages={192--217},
   review={ Zbl 46.0547.01},
}

\bib{montiel-1}{article}{
   author={Montiel, Sebasti\'{a}n},
   title={Stable constant mean curvature hypersurfaces in some Riemannian
   manifolds},
   journal={Comment. Math. Helv.},
   volume={73},
   date={1998},
   number={4},
   pages={584--602},
   issn={0010-2571},
   review={\MR{1639892}},
   doi={10.1007/s000140050070},
}
\bib{Montiel}{article}{
   author={Montiel, Sebasti{\'a}n},
   title={Unicity of constant mean curvature hypersurfaces in some
   Riemannian manifolds},
   journal={Indiana Univ. Math. J.},
   volume={48},
   date={1999},
   number={2},
   pages={711--748},
   issn={0022-2518},
   review={\MR{1722814 (2001f:53131)}},
   doi={10.1512/iumj.1999.48.1562},
}
\bib{pacard}{article}{
   author={Pacard, F.},
   author={Xu, X.},
   title={Constant mean curvature spheres in Riemannian manifolds},
   journal={Manuscripta Math.},
   volume={128},
   date={2009},
   number={3},
   pages={275--295},
   issn={0025-2611},
   review={\MR{2481045}},
   doi={10.1007/s00229-008-0230-7},
}

\bib{petersen}{book}{
   author={Petersen, Peter},
   title={Riemannian geometry},
   series={Graduate Texts in Mathematics},
   volume={171},
   edition={2},
   publisher={Springer, New York},
   date={2006},
   pages={xvi+401},
   isbn={978-0387-29246-5},
   isbn={0-387-29246-2},
   review={\MR{2243772}},
}
\bib{ritore}{article}{
   author={Ritor\'{e}, Manuel},
   title={Constant geodesic curvature curves and isoperimetric domains in
   rotationally symmetric surfaces},
   journal={Comm. Anal. Geom.},
   volume={9},
   date={2001},
   number={5},
   pages={1093--1138},
   issn={1019-8385},
   review={\MR{1883725}},
   doi={10.4310/CAG.2001.v9.n5.a5},
}
\bib{sal-sal}{article}{
   author={Salamanca, Juan J.},
   author={Salavessa, Isabel M. C.},
   title={Uniqueness of $\phi$-minimal hypersurfaces in warped
   product manifolds},
   journal={J. Math. Anal. Appl.},
   volume={422},
   date={2015},
   number={2},
   pages={1376--1389},
   issn={0022-247X},
   review={\MR{3269517}},
   doi={10.1016/j.jmaa.2014.09.028},
}

\bib{SN}{article}{
   author={Silva Neto, Greg\'{o}rio},
   title={Stability of constant mean curvature surfaces in three-dimensional
   warped product manifolds},
   journal={Ann. Global Anal. Geom.},
   volume={56},
   date={2019},
   number={1},
   pages={57--86},
   issn={0232-704X},
   review={\MR{3962026}},
   doi={10.1007/s10455-019-09656-x},
}
\bib{IT}{book}{
   author={Imayoshi, Y.},
   author={Taniguchi, M.},
   title={An introduction to Teichm\"{u}ller spaces},
   note={Translated and revised from the Japanese by the authors},
   publisher={Springer-Verlag, Tokyo},
   date={1992},
   pages={xiv+279},
   isbn={4-431-70088-9},
   review={\MR{1215481}},
   doi={10.1007/978-4-431-68174-8},
}


\bib{Xia-Wu-1}{article}{
   author={Wu, Jie},
   author={Xia, Chao},
   title={On rigidity of hypersurfaces with constant curvature functions in
   warped product manifolds},
   journal={Ann. Global Anal. Geom.},
   volume={46},
   date={2014},
   number={1},
   pages={1--22},
   issn={0232-704X},
   review={\MR{3205799}},
   doi={10.1007/s10455-013-9405-x},
}

\bib{Xia-Wu-2}{article}{
   author={Wu, Jie},
   author={Xia, Chao},
   title={Hypersurfaces with constant curvature quotients in warped product
   manifolds},
   journal={Pacific J. Math.},
   volume={274},
   date={2015},
   number={2},
   pages={355--371},
   issn={0030-8730},
   review={\MR{3332908}},
   doi={10.2140/pjm.2015.274.355},
}
\end{biblist}
\end{bibdiv}
\end{document}